\documentclass{amsart}
\usepackage{setspace}
\usepackage{a4}
\usepackage{amssymb,amsmath,amsthm,latexsym}
\usepackage{amsfonts}
\usepackage{amsfonts}
\usepackage{graphicx}
\usepackage{textcomp}
\usepackage{cite}
\usepackage{enumerate}
\usepackage[mathscr]{euscript}
\usepackage{mathtools}
\newtheorem{theorem}{Theorem}[section]

\newtheorem{conjecture}[theorem]{Conjecture}

\newtheorem{definition}[theorem]{Definition}

\title{This is the title}
\usepackage{amssymb}
\usepackage{amssymb}
\usepackage{amssymb}
\usepackage{amssymb}
\usepackage{amsmath}
\usepackage{tikz}
\usepackage{hyperref}
\usepackage{enumerate}
\usepackage{mathtools}
\usepackage{amsmath}
\usepackage{tikz}
\usepackage{amssymb}
\usepackage{amsmath}
\usepackage{tikz}
\usepackage{hyperref}
\usepackage{fancyhdr}
\begin{document}
	\begin{center}
	{\bf{C*-ALGEBRAIC GAUSS-LUCAS THEOREM AND C*-ALGEBRAIC SENDOV'S CONJECTURE}}\\
	{\bf{K. MAHESH KRISHNA}}\\
	Post Doctoral Fellow\\
	Statistics and Mathematics Unit\\
	Indian Statistical Institute, Bangalore Centre\\
	Karnataka 560 059 India\\
	Email: kmaheshak@gmail.com \\
	\today
\end{center}

\hrule
\vspace{0.5cm}

\textbf{Abstract}:  Using a result of Robertson \textit{[Proc. Edinburgh Math. Soc. (2), 1976]}, we introduce a  notion of differentiation of maps on certain classes of unital commutative  C*-algebras. We then derive C*-algebraic Gauss-Lucas theorem and formulate C*-algebraic Sendov's conjecture. We verify C*-algebraic Sendov's conjecture for polynomials of degree 2.  


\tableofcontents
\section{Introduction}
Let $a$ be a complex number, $r>0$ be a real number and $\overline{\mathbb{D}(a, r)}$ be the closed disc in $\mathbb{C}$ centered at $a$ of radius $r$. Let $\mathbb{C}[z]$ be the set of all polynomials over $\mathbb{C}$. In 1958, Sendov made the following conjecture which became known as Sendov's conjecture.
 \begin{conjecture}\label{SENDOV} \cite{HAYMAN, MARDEN, RAHMANSCHMEISSERBOOK}	\textbf{(Sendov's conjecture)
 	Let $n \in \mathbb{N}\setminus \{1\}$ and $p(z)=(z-a_1)(z-a_2)\cdots (z-a_n)\in \mathbb{C}[z]$ be such that $a_1, a_2, \dots,  a_n \in \overline{\mathbb{D}(0, 1)}$. Then for each $a_j$, $1\leq j\leq n$, there exists  a zero $b$ of $p'$ such that $b \in  \overline{\mathbb{D}(a_j, 1)}$.} 
 \end{conjecture} 
   A direct calculation says that Conjecture \ref{SENDOV} holds for degree two polynomials. For $n=3$, Conjecture  \ref{SENDOV}  is proved in \cite{BRANNAN}. For $n=3,4$,  Conjecture \ref{SENDOV}  is proved in \cite{RUBINSTEIN}.   For $n=5$,  Conjecture \ref{SENDOV}  is proved in \cite{MEIRSHARMA}. For $n=6$,  Conjecture \ref{SENDOV}  is proved in \cite{KATSOPRINAKIS, BORCEA1996}. For $n=7$,  Conjecture \ref{SENDOV}  is proved in \cite{BROWN1997}. For $n=8$,  Conjecture \ref{SENDOV}  is proved in \cite{BROWNXIANG}. Conjecture \ref{SENDOV}  is proved for all polynomials after certain (unknown) degree in \cite{TAO}. On the other hand, Conjecture \ref{SENDOV}  is proved for various special cases of polynomials and position of roots (inside the closed unit disc) in \cite{BORCEA, BROWN1991, BROWN, BOJANOV, SCHMEISSER, DEGOT, CHALEBGWA, RUBINSTEIN, GOODMANRAHMANRATTI, MILLER, CHIJIWA, KASMALKAR, JOYAL, VAIJAITUZAHARESCU, MILLER2005, MILLER2008, SAFFTWOMEY, MILLER1990, BORCEA1998, RAHMANTARIQ, GOODMAN, PAWLOWSKI, MILLER1992, CHIJIWA2010, SCHMIEDERSZYNAL, DAEPPGORKINVOSS, BOJANOVRAHMANSZYNAL, BYRNE, BYRNE1997}. Non-Archimedean version of Conjecture \ref{SENDOV} has been solved in  \cite{CHOILEE}. In this paper, we derive Gauss-Lucas theorem for polynomials over certain C*-algebras. We then formulate Sendov's conjecture for polynomials over C*-algebras.

\section{C*-algebraic Gauss-Lucas Theorem and C*-algebraic Sendov's Conjecture}
Let $\mathcal{A}$	be a unital commutative C*-algebra and let $G(\mathcal{A})$ be the set of all invertible elements of $\mathcal{A}$. Let $\mathcal{A}[z]$ be the set of all polynomials over $\mathcal{A}$. We assume $G(\mathcal{A})$ is dense in $\mathcal{A}$. Such C*-algebras exist and a characterization of such C*-algebras is given by Robertson  \cite{ROBERTSON}. We now introduce the following definition.
\begin{definition}\label{CDIFFE}
\textbf{(C*-algebraic differentiation) Let $\mathcal{A}$	be a unital commutative C*-algebra and $G(\mathcal{A})$ be  dense in $\mathcal{A}$. Let $f: \mathcal{A}\to \mathcal{A}$ be a function and $\omega\in \mathcal{A}$. We say that	$f$ is C*-algebraic differentiable at $\omega$ if there exists an $L\in \mathcal{A}$ satisfying the following:  for each $\varepsilon>0$, there exists a $\delta>0$ such that if $z\in \mathcal{A}$ satisfies $\|z-\omega\|<\delta$ and $z-\omega\in G(\mathcal{A})$, then 
\begin{align*}
	\|(z-\omega)^{-1}(f(z)-f(\omega))-L\|<\varepsilon.
\end{align*}}
In this case, we write $f'(\omega)=L$.
\end{definition}
We note that, because $G(\mathcal{A})$ is  dense in $\mathcal{A}$, Definition \ref{CDIFFE} is well-defined. Now we derive C*-algebraic Gauss-Lucas theorem. 
\begin{theorem}
\textbf{(C*-algebraic Gauss-Lucas Theorem)} Let $p(z)=(z-a_1)(z-a_2)\cdots (z-a_n)\in\mathcal{A}[z]$. 	Define 
\begin{align*}
	\mathcal{B}\coloneqq \{z \in \mathcal{A}: z-a_j \text{ is  invertible in }  \mathcal{A} \text{ for all } 1\leq j \leq n\}, \quad 	\mathcal{C}\coloneqq \{z \in \mathcal{A}: p'(z)=0\}.	
\end{align*}
 Then for each $z \in \mathcal{B}\cap  \mathcal{C}$, there are positive $\omega_{z_1}, \dots, \omega_{z_n} \in  \mathcal{A}$ such that 
\begin{align}\label{CONVEX}
	z=\sum_{j=1}^{n}\omega_{z_j}a_j, \quad \sum_{j=1}^{n}\omega_{z_j}=1.
\end{align}
\end{theorem}
\begin{proof}
	We have 
	\begin{align*}
		p'(z)=\sum_{j=1}^{n}(z-a_1)\cdots \widehat{(z-a_j)}\cdots (z-a_n), \quad \forall z \in \mathcal{A}.
	\end{align*}
	where the term with cap is missing. Then 
	\begin{align*}
		p(z)^{-1}	p'(z)=\sum_{j=1}^{n}(z-a_j)^{-1}, \quad \forall z \in \mathcal{B}.
	\end{align*}
We therefore have 
\begin{align}\label{FIRST}
	0=\sum_{j=1}^{n}(z-a_j)^{-1}=\sum_{j=1}^{n}((z-a_j)(z-a_j)^*)^{-1}(z^*-a_j^*), \quad \forall z \in \mathcal{B}\cap  \mathcal{C}.
\end{align}
Now  define 
\begin{align*}
	\omega_z \coloneqq \sum_{j=1}^{n}((z-a_j)(z-a_j)^*)^{-1}, \quad \forall z \in \mathcal{B}\cap  \mathcal{C}.
\end{align*}
Note that whenever $ z \in \mathcal{B}$, each $((z-a_j)(z-a_j)^*)^{-1}$ is positive and invertible, $1\leq j \leq n$. Hence $\omega_z$ is positive and invertible. Equation (\ref{FIRST}) then gives 
\begin{align*}
z=\omega_z^{-1}\sum_{j=1}^{n}((z-a_j)(z-a_j)^*)^{-1}a_j, \quad \forall z \in \mathcal{B}\cap  \mathcal{C}.
\end{align*}
\end{proof}
\begin{definition}\label{CDISC}
	\textbf{Given a unital  C*-algebra $\mathcal{A}$ with identity $1$  and an element $a\in \mathcal{A}$, we define the C*-algebraic closed unit disc centered at $a$ and of radius $r>0$, $r\in \mathbb{R}$, denoted as $\overline{\mathbb{D}^*(a, r)}$  by
	\begin{align*}
	 \overline{\mathbb{D}^*(a, r)}\coloneqq 	\{z\in \mathcal{A}: (z-a)(z-a)^*\leq \sqrt{r}\cdot 1\}.
	\end{align*}}
\end{definition}
With Definition \ref{CDISC} we can formulate Conjecture \ref{SENDOV} for C*-algebras.
\begin{conjecture}\label{CSENDOV} 	\textbf{(Commutative C*-algebraic Sendov's conjecture)
Let $\mathcal{A}$	be a unital commutative C*-algebra and $G(\mathcal{A})$ be  dense in $\mathcal{A}$. 	Let $n \in \mathbb{N}\setminus \{1\}$ and $p(z)=(z-a_1)(z-a_2)\cdots (z-a_n)\in\mathcal{A}[z]$ be such that $a_1, a_2, \dots,  a_n \in \overline{\mathbb{D^*}(0, 1)}$. Assume that $p'$ admits roots in $\mathcal{A}$, say $b_1, b_2, \dots,  b_{n-1} \in \overline{\mathbb{D^*}(0, 1)}$ and each $b_k$ can be written in the form of Equation (\ref{CONVEX}). Then for each $a_j$, $1\leq j\leq n$, there exists  a zero $b$ of $p'$ such that $b \in  \overline{\mathbb{D}^*(a_j, 1)}$.} 	
\end{conjecture}
For arbitary C*-algebras, without introducing differentiation we can formulate Conjecture Conjecture \ref{CSENDOV}   as follows. 
\begin{conjecture}\label{CSENDOV2} 
\textbf{(C*-algebraic Sendov's conjecture)
	Let $\mathcal{A}$	be a unital  C*-algebra. 	Let $n \in \mathbb{N}\setminus \{1\}$ and $p(z)=(z-a_1)(z-a_2)\cdots (z-a_n)\in\mathcal{A}[z]$ be such that $a_1, a_2, \dots,  a_n \in \overline{\mathbb{D^*}(0, 1)}$.  Define 
	\begin{align*}
		p'(z)=\sum_{j=1}^{n}(z-a_1)\cdots \widehat{(z-a_j)}\cdots (z-a_n), \quad \forall z \in \mathcal{A}.
	\end{align*}
	where the term with cap is missing. Assume that $p'$ admits roots in $\mathcal{A}$, say $b_1, b_2, \dots,  b_{n-1} \in \overline{\mathbb{D^*}(0, 1)}$ and each $b_k$ can be written in the form of Equation (\ref{CONVEX}). Then for each $a_j$, $1\leq j\leq n$, there exists  a zero $b$ of $p'$ such that $b \in  \overline{\mathbb{D}^*(a_j, 1)}$.} 	
\end{conjecture}
We end with the following result.
\begin{theorem}
Let $\mathcal{A}$	be a unital  C*-algebra. Conjecture \ref{CSENDOV2}   holds for polynomials  of degree 2 of the form $p(z)=(z-a)(z-b)\in \mathcal{A}[z]$, where $a, b \in \overline{\mathbb{D}^*(0,1)}$. 
\end{theorem}
\begin{proof}
We have 
\begin{align*}
	p'(z)=2z-(a+b), \quad \forall z \in \mathcal{A}.
\end{align*}
Therefore the zero of $p'$ is $\frac{a+b}{2}$ which is in the form of Equation (\ref{CONVEX}). Now we have
\begin{align*}
	2(aa^*+bb^*)\geq (a+b)(a+b)^*, \quad 2(aa^*+bb^*)\geq (a-b)(a-b)^* 
\end{align*}
which gives
\begin{align*}
\left(\frac{a+b}{2}-a\right)\left(\frac{a+b}{2}-a\right)^*\leq 1, \quad \left(\frac{a+b}{2}-b\right)\left(\frac{a+b}{2}-b\right)^*\leq 1.	
\end{align*}
Therefore 
\begin{align*}
\frac{a+b}{2}\in \overline{\mathbb{D}^*\left(a, 1\right)}, \quad \frac{a+b}{2}\in \overline{\mathbb{D}^*(b, 1)}	
\end{align*}
which completes the proof. 
\end{proof}

 \bibliographystyle{plain}
 \bibliography{reference.bib}

\begin{thebibliography}{10}

\bibitem{BOJANOVRAHMANSZYNAL}
B.~D. Bojanov, Q.~I. Rahman, and J.~Szynal.
\newblock On a conjecture of {S}endov about the critical points of a
  polynomial.
\newblock {\em Math. Z.}, 190(2):281--285, 1985.

\bibitem{BOJANOV}
Borislav Bojanov.
\newblock Extremal problems for polynomials in the complex plane.
\newblock In {\em Approximation and computation}, volume~42 of {\em Springer
  Optim. Appl.}, pages 61--85. Springer, New York, 2011.

\bibitem{BORCEA1996}
Iulius Borcea.
\newblock On the {S}endov conjecture for polynomials with at most six distinct
  roots.
\newblock {\em J. Math. Anal. Appl.}, 200(1):182--206, 1996.

\bibitem{BORCEA}
Iulius Borcea.
\newblock The {S}endov conjecture for polynomials with at most seven distinct
  zeros.
\newblock {\em Analysis}, 16(2):137--159, 1996.

\bibitem{BORCEA1998}
Julius Borcea.
\newblock Two approaches to {S}endov's conjecture.
\newblock {\em Arch. Math. (Basel)}, 71(1):46--54, 1998.

\bibitem{BRANNAN}
D.~A. Brannan.
\newblock On a conjecture of {I}lieff.
\newblock {\em Proc. Cambridge Philos. Soc.}, 64:83--85, 1968.

\bibitem{BROWN}
Johnny~E. Brown.
\newblock On the {I}lieff-{S}endov conjecture.
\newblock {\em Pacific J. Math.}, 135(2):223--232, 1988.

\bibitem{BROWN1991}
Johnny~E. Brown.
\newblock On the {S}endov conjecture for sixth degree polynomials.
\newblock {\em Proc. Amer. Math. Soc.}, 113(4):939--946, 1991.

\bibitem{BROWN1997}
Johnny~E. Brown.
\newblock A proof of the {S}endov conjecture for polynomials of degree seven.
\newblock {\em Complex Variables Theory Appl.}, 33(1-4):75--95, 1997.

\bibitem{BROWNXIANG}
Johnny~E. Brown and Guangping Xiang.
\newblock Proof of the {S}endov conjecture for polynomials of degree at most
  eight.
\newblock {\em J. Math. Anal. Appl.}, 232(2):272--292, 1999.

\bibitem{BYRNE}
Angelina Byrne.
\newblock Some results for the {S}endov conjecture.
\newblock {\em J. Math. Anal. Appl.}, 199(3):754--768, 1996.

\bibitem{BYRNE1997}
Angelina Byrne.
\newblock Results pertaining to the {S}endov conjecture.
\newblock {\em J. Math. Anal. Appl.}, 212(2):333--342, 1997.

\bibitem{CHALEBGWA}
T.~P. Chalebgwa.
\newblock Sendov's conjecture: a note on a paper of {D}\'{e}got.
\newblock {\em Anal. Math.}, 46(3):447--463, 2020.

\bibitem{CHIJIWA2010}
Tomohiro Chijiwa.
\newblock A quantitative result on polynomials with zeros in the unit disk.
\newblock {\em Proc. Japan Acad. Ser. A Math. Sci.}, 86(10):165--168, 2010.

\bibitem{CHIJIWA}
Tomohiro Chijiwa.
\newblock A quantitative result on {S}endov's conjecture for a zero near the
  unit circle.
\newblock {\em Hiroshima Math. J.}, 41(2):235--273, 2011.

\bibitem{CHOILEE}
Daebeom Choi and Seewoo Lee.
\newblock Non-{A}rchimedean {S}endov's {C}onjecture.
\newblock {\em p-Adic Numbers Ultrametric Anal. Appl.}, 14(1):77--80, 2022.

\bibitem{DAEPPGORKINVOSS}
Ulrich Daepp, Pamela Gorkin, and Karl Voss.
\newblock Poncelet's theorem, {S}endov's conjecture, and {B}laschke products.
\newblock {\em J. Math. Anal. Appl.}, 365(1):93--102, 2010.

\bibitem{DEGOT}
J.~Degot.
\newblock Sendov conjecture for high degree polynomials.
\newblock {\em Proc. Amer. Math. Soc.}, 142(4):1337--1349, 2014.

\bibitem{GOODMAN}
A.~W. Goodman.
\newblock On the derivative with respect to a point.
\newblock {\em Proc. Amer. Math. Soc.}, 101(2):327--330, 1987.

\bibitem{GOODMANRAHMANRATTI}
A.~W. Goodman, Q.~I. Rahman, and J.~S. Ratti.
\newblock On the zeros of a polynomial and its derivative.
\newblock {\em Proc. Amer. Math. Soc.}, 21:273--274, 1969.

\bibitem{HAYMAN}
W.~K. Hayman.
\newblock {\em Research problems in function theory}.
\newblock The Athlone Press London, 1967.

\bibitem{JOYAL}
Andr\'{e} Joyal.
\newblock On the zeros of a polynomial and its derivative.
\newblock {\em J. Math. Anal. Appl.}, 26:315--317, 1969.

\bibitem{KASMALKAR}
Indraneel~G. Kasmalkar.
\newblock On the {S}endov conjecture for a root close to the unit circle.
\newblock {\em Aust. J. Math. Anal. Appl.}, 11(1):Art. 4, 34, 2014.

\bibitem{KATSOPRINAKIS}
Emmanuel~S. Katsoprinakis.
\newblock Erratum to: ``{O}n the {S}endov-{I}lyeff conjecture'' [{B}ull.
  {L}ondon {M}ath. {S}oc. {\bf 24} (1992), no. 5, 449--455; {MR}1173941
  (93f:30005)].
\newblock {\em Bull. London Math. Soc.}, 28(6):605--612, 1996.

\bibitem{MARDEN}
Morris Marden.
\newblock Conjectures on the critical points of a polynomial.
\newblock {\em Amer. Math. Monthly}, 90(4):267--276, 1983.

\bibitem{MEIRSHARMA}
A.~Meir and A.~Sharma.
\newblock On {I}lyeff's conjecture.
\newblock {\em Pacific J. Math.}, 31:459--467, 1969.

\bibitem{MILLER1990}
Michael~J. Miller.
\newblock Maximal polynomials and the {I}lieff-{S}endov conjecture.
\newblock {\em Trans. Amer. Math. Soc.}, 321(1):285--303, 1990.

\bibitem{MILLER1992}
Michael~J. Miller.
\newblock Continuous independence and the {I}lieff-{S}endov conjecture.
\newblock {\em Proc. Amer. Math. Soc.}, 115(1):79--83, 1992.

\bibitem{MILLER}
Michael~J. Miller.
\newblock On {S}endov's conjecture for roots near the unit circle.
\newblock {\em J. Math. Anal. Appl.}, 175(2):632--639, 1993.

\bibitem{MILLER2005}
Michael~J. Miller.
\newblock A quadratic approximation to the {S}endov radius near the unit
  circle.
\newblock {\em Trans. Amer. Math. Soc.}, 357(3):851--873, 2005.

\bibitem{MILLER2008}
Michael~J. Miller.
\newblock Unexpected local extrema for the {S}endov conjecture.
\newblock {\em J. Math. Anal. Appl.}, 348(1):461--468, 2008.

\bibitem{PAWLOWSKI}
Piotr Pawlowski.
\newblock On the zeros of a polynomial and its derivatives.
\newblock {\em Trans. Amer. Math. Soc.}, 350(11):4461--4472, 1998.

\bibitem{RAHMANSCHMEISSERBOOK}
Q.~I. Rahman and G.~Schmeisser.
\newblock {\em Analytic theory of polynomials}, volume~26 of {\em London
  Mathematical Society Monographs. New Series}.
\newblock The Clarendon Press, Oxford University Press, Oxford, 2002.

\bibitem{RAHMANTARIQ}
Q.~I. Rahman and Q.~M. Tariq.
\newblock On a problem related to the conjecture of {S}endov about the critical
  points of a polynomial.
\newblock {\em Canad. Math. Bull.}, 30(4):476--480, 1987.

\bibitem{ROBERTSON}
Guyan Robertson.
\newblock On the density of the invertible group in {$C\sp*$}-algebras.
\newblock {\em Proc. Edinburgh Math. Soc. (2)}, 20(2):153--157, 1976.

\bibitem{RUBINSTEIN}
Zalman Rubinstein.
\newblock On a problem of {I}lyeff.
\newblock {\em Pacific J. Math.}, 26:159--161, 1968.

\bibitem{SAFFTWOMEY}
E.~B. Saff and J.~B. Twomey.
\newblock A note on the location of critical points of polynomials.
\newblock {\em Proc. Amer. Math. Soc.}, 27:303--308, 1971.

\bibitem{SCHMEISSER}
Gerhard Schmeisser.
\newblock Bemerkungen zu einer {V}ermutung von {I}lieff.
\newblock {\em Math. Z.}, 111:121--125, 1969.

\bibitem{SCHMIEDERSZYNAL}
Gerald Schmieder and Jan Szynal.
\newblock On the distribution of the derivative zeros of a complex polynomial.
\newblock {\em Complex Var. Theory Appl.}, 47(3):239--241, 2002.

\bibitem{TAO}
Terence Tao.
\newblock Sendov's conjecture for sufficiently high degree polynomials.
\newblock {\em https://arxiv.org/abs/2012.04125}, 2020.

\bibitem{VAIJAITUZAHARESCU}
V.~Vajaitu and A.~Zaharescu.
\newblock Ilyeff's conjecture on a corona.
\newblock {\em Bull. London Math. Soc.}, 25(1):49--54, 1993.

\end{thebibliography}

\end{document}